\documentclass[12pt,reqno]{amsart}
\usepackage[utf8]{inputenc}
\usepackage{amsmath}
\usepackage{amssymb}
\usepackage{amsthm}
\usepackage{xcolor}
\usepackage{fourier}
\usepackage{fullpage}
\usepackage{hyperref}
\usepackage{tikz}
\usepackage{soul}
\usepackage{mathrsfs}
\usepackage{enumitem}
\usepackage{bm}
\usepackage{mathtools}
\usepackage{amstext,amsgen,amsfonts,verbatim}
\usepackage{microtype}

\newcommand{\cH}{{\mathcal H}}

\newcommand{\R}{{\mathbb{R}}}

\DeclareMathOperator{\diam}{diam}

\newcommand{\HH}{\mathbb H}
\newcommand{\E}{\mathbb E}
\newcommand{\cD}{\mathcal D}
\newcommand{\1}{\mathbf 1}

\newcommand{\energy}{\mathcal E}
\newcommand{\norm}[1]{\left\lVert #1\right\rVert}

\numberwithin{equation}{section}
\newtheorem{theorem}{Theorem}[section]
\newtheorem{proposition}[theorem]{Proposition}
\newtheorem{lemma}[theorem]{Lemma}
\newtheorem{definition}[theorem]{Definition}
\newtheorem{corollary}[theorem]{Corollary}
\theoremstyle{remark}
\newtheorem{remark}[theorem]{Remark}

\title{Sharp boundedness criteria for nonnegative kernels on regular curves in the Heisenberg group}
\author{Lingxiao Zhang}
\address{Center for Applied Mathematics and KL-AAGDM, Tianjin University, Tianjin, 300072, China}
\email{lingxiaoz@tju.edu.cn}

\begin{document}
\begin{abstract}
In this paper we study singular integrals associated with the nonnegative kernels
\[
K_\alpha(x,y,z)=\frac{|z|^{\alpha/2}}{\|(x,y,z)\|^{\alpha+1}}, \qquad \alpha>0,
\]
in the first Heisenberg group $\mathbb H$, where $\|\cdot\|$ denotes the Korányi norm. This family was introduced by Chousionis and Li.

We first show that there exists a compact $1$-Ahlfors regular horizontal $C^1$ curve $E\subset\mathbb H$ such that the singular integrals associated with $K_\alpha$ are unbounded on $L^2(E)$ for every $0<\alpha<4$. This extends the counterexamples of Chousionis, Li and Zhang from $0<\alpha<2$ to the full range $0<\alpha<4$. Combined with the boundedness theorem of Fässler and Orponen for $\alpha\ge4$, our result determines the exact range of exponents for boundedness on every $1$-Ahlfors regular curve.

We also consider the more general kernels
\[
K_\Phi(p)=\frac{1}{\|p\|} \Phi\!\left(\frac{|z|^{1/2}}{\|p\|}\right), \qquad p=(x,y,z)\neq0,
\]
where $\Phi:[0,1]\to[0,\infty)$ is continuous and nondecreasing, with $\Phi(0)=0$. We show that the associated singular integral is $L^2$-bounded on every $1$-Ahlfors regular curve if and only if $\Phi(u)=O(u^4)$ as $u\to0$. When this condition fails, we construct a compact $1$-Ahlfors regular horizontal $C^1$ curve on which the operator is unbounded. The curves in both constructions can be chosen as horizontal lifts of planar $C^1$ graphs with arbitrarily small slope.
\end{abstract}
\maketitle

\section{Introduction and preliminaries}

The relation between the geometry of a set and the boundedness of singular integral operators is a central topic in harmonic analysis and geometric measure theory. For the Cauchy transform, fundamental links with curvature and rectifiability were established by Melnikov and Verdera \cite{MV1995} and Mattila, Melnikov, and Verdera \cite{MMV1996}. Cancellation plays an important role in many of the classical results. For example, the kernel $|s-t|^{-1}$ does not define a bounded singular integral on an interval, since its integral diverges near $s=t$. In the Heisenberg group, Chousionis and Li~\cite{CL} showed that some nonnegative kernels can nevertheless define bounded singular integrals on every $1$-Ahlfors regular curve. In this paper we determine the exact range of exponents for their family of kernels to be bounded and prove a more general necessary and sufficient condition for boundedness.

The first Heisenberg group $\HH$ is $\R^3$ equipped with the group operation
\[
 (x,y,z)\cdot(x',y',z') =\bigl(x+x',y+y',z+z'+\tfrac12(xy'-yx')\bigr).
\]
We use the homogeneous norm
\[
 \|(x,y,z)\|=\bigl((x^2+y^2)^2+z^2\bigr)^{1/4}
\]
and the left invariant metric $d(p,q)=\|p^{-1}q\|$. The dilations are $\delta_r(x,y,z)=(rx,ry,r^2z)$, for $r>0$, and satisfy $d(\delta_r p,\delta_r q)=r\,d(p,q)$. Throughout the paper, $\cH^1$ denotes the one-dimensional Hausdorff measure with respect to $d$.

A Borel set $E\subset\HH$ is called $1$-Ahlfors regular if there is a constant $C\ge1$ such that
\[
 C^{-1}r\le\cH^1(E\cap B(p,r))\le Cr \qquad\text{for all }p\in E,\quad 0<r\le\diam E.
\]
We write $L^2(E)=L^2(\cH^1|_E)$.

\begin{definition}
\label{czkernel}
A continuous function $K: \mathbb{H}\backslash \{0\} \to \mathbb{R}^d$ is called a \textit{$1$-dimensional Calder\'on-Zygmund (CZ) kernel} in $\mathbb{H}$ if there exist $\kappa \in (0,1),\beta \in (0,1], C_K \ge 1,$ such that:
\begin{itemize}
\item (Growth condition) 
\begin{equation}\label{czgrowth}
|K(p)| \le C_K \frac{1}{\|p\|}, \quad \forall p\in \mathbb{H}\backslash \{0\},
\end{equation}
    
\item (H\"older continuity)
\begin{equation}\label{czhol}
|K(p_1)-K(p_2)| \le C_K \frac{\|p_2^{-1} \cdot p_1\|^\beta}{\|p_1\|^{1+\beta}}, \quad \forall p_1, p_2\in \mathbb{H}\backslash \{0\} \text{ with } d(p_1, p_2) \le \kappa \|p_1\|.
\end{equation}
\end{itemize}
\end{definition}

If $K$ is a $1$-dimensional CZ kernel in $\HH$ and $E\subset \HH$ is a $1$-Ahlfors regular set, then the truncated singular integrals
\[
 T_{K,E,\varepsilon}f(p) =\int_{\{q\in E:\,d(p,q)>\varepsilon\}} K(p^{-1}q)f(q)\,d\cH^1(q),\qquad p\in E,
\]
are well-defined for $f\in L^2(E)$ and $\varepsilon>0$ when $p\in (1,\infty)$. We say that the associated singular integral operator $T_{K,E}$ is bounded on $L^2(E)$ if there is a constant $C$ such that
\[
 \norm{T_{K,E,\varepsilon}f}_{L^2(E)} \le C\norm{f}_{L^2(E)} \qquad\text{for every }\varepsilon>0\text{ and }f\in L^2(E).
\]
The constant is independent of $\varepsilon$, but may depend on $K$ and $E$.

Chousionis and Li~\cite{CL} introduced the family
\begin{equation}\label{eq:powers}
 K_\alpha(p)=\frac{|z|^{\alpha/2}}{\|p\|^{\alpha+1}}, \qquad p=(x,y,z)\neq0,\quad\alpha>0.
\end{equation}
These kernels are nonnegative and $-1$-homogeneous, i.e., $K_\alpha(\delta_r p)=r^{-1}K_\alpha(p)$. They are also $1$-dimensional CZ kernels. A horizontal line is a left translate of $\{(ta,tb,0):t\in\R\}$, where $(a,b)\neq(0,0)$. If $p$ and $q$ belong to the same horizontal line, the third coordinate of $p^{-1}q$ is zero. Consequently, $K_\alpha(p^{-1}q)=0$ whenever $p\neq q$ lie on that line. Thus the associated operators vanish on horizontal lines for every $\alpha>0$.

In~\cite{CL}, Chousionis and Li proved that $T_{K_8,E}$ is bounded on $L^2(E)$ whenever $E$ is a $1$-Ahlfors regular curve. They also showed that, for a $1$-Ahlfors regular set $E$, boundedness of $T_{K_2,E}$ implies that $E$ is contained in a regular curve. F\"assler and Orponen improved the first result by proving boundedness on regular curves for every $\alpha\ge4$; see~\cite[Theorem 1.7]{FO}. More recently, Chousionis, Li and Zhang~\cite{CLZ26} proved that the boundedness of $T_{K_4, E}$ also implies that $E$ is contained in a regular curve. Together, these results characterize the $1$-Ahlfors regular sets contained in regular curves by the $L^2$-boundedness of the operator associated with $K_4$.

It is therefore natural to ask whether $K_\alpha$ remains bounded on every regular curve when $\alpha<4$. F\"assler and Orponen raised this question for $\alpha\ge2$ after their Theorem 1.7. Chousionis, Li and Zhang~\cite[Theorem 1.5]{CLZ26} constructed a regular curve $E_\alpha$ on which $T_{K_\alpha,E_\alpha}$ is unbounded for each $0<\alpha<2$. The remaining range is $2\le\alpha<4$. We show that the exponent $4$ is sharp. In fact, there exists a compact $1$-Ahlfors regular horizontal $C^1$ curve for which unboundedness holds for every $0<\alpha<4$ simultaneously.

We obtain this result from a more general construction. We first describe the curves that we will use. For $g\in L^\infty(0,1)$, set
\begin{equation}\label{eq:lift}
\begin{aligned}
f_g(t)&=\int_0^t g(u)\,du, &z_g(t)&=\frac12\int_0^t\bigl(ug(u)-f_g(u)\bigr)\,du,\\
\gamma_g(t)&=(t,f_g(t),z_g(t)),&&\qquad 0\le t\le1.
\end{aligned}
\end{equation}
Since $g$ is bounded and measurable, $f_g$ and $z_g$ are continuous. Note $\gamma_g$ is the horizontal lift of the planar graph $t\mapsto(t,f_g(t))$. Indeed, its coordinates satisfy $z_g'(t)=\tfrac12\bigl(t f_g'(t)-f_g(t)\bigr)$ almost everywhere. When $\norm{g}_\infty\le1$, the map $\gamma_g$ is bi-Lipschitz, with absolute constants, and its image is a compact regular curve. These facts are proved in Lemma~\ref{lem:lift}. If $g$ is continuous, then $\gamma_g$ is also $C^1$ in the Euclidean coordinates and has nonzero derivative. In this case, $\norm{g}_\infty$ is the maximum absolute value of the slope of the planar graph. We refer to $g$ as the planar slope, since $f_g'=g$ almost everywhere.

\begin{definition}
For a compact $1$-Ahlfors regular curve $E$, put $\mu=\cH^1|_E$ and define the (possibly infinite) energy
\[
 \energy_K(E)=\lim_{\varepsilon\downarrow0} \iint_{\substack{(p,q)\in E\times E\\d(p,q)>\varepsilon}} K(p^{-1}q)\,d\mu(q)\,d\mu(p).
\]
The limit exists because the kernel is nonnegative.
\end{definition}

If $\energy_K(E)=\infty$, then $T_{K,E}$ is unbounded on $L^2(E)$: a uniform operator bound, tested against the constant function $1$ on $E$, would give a finite upper bound for all the truncated energies.

Let $\Phi:[0,1]\to[0,\infty)$ be continuous and nondecreasing, with $\Phi(0)=0$. Consider
\begin{equation}\label{eq:profile}
\theta(p)=\frac{\sqrt{|z|}}{\|p\|},\qquad K_\Phi(p)=\frac{\Phi(\theta(p))}{\|p\|}, \qquad p=(x,y,z)\neq0.
\end{equation}
Since $0\le\theta(p)\le1$ and $\theta(\delta_r p)=\theta(p)$, the kernel $K_\Phi$ is $-1$-homogeneous. The choice $\Phi(u)=u^\alpha$ gives $K_\Phi=K_\alpha$. For general $\Phi$ we impose continuity, but do not require a H\"older continuity condition. Our first theorem gives an exact criterion for boundedness within this family.

\begin{theorem}\label{thm:classification}
Let $\Phi:[0,1]\to[0,\infty)$ be continuous and nondecreasing, with $\Phi(0)=0$, and let $K_\Phi$ be defined by~\eqref{eq:profile}. The following are equivalent:
\begin{enumerate}[label=\textup{(\roman*)}]
\item $\Phi(u)=O(u^4)$ as $u\downarrow0$.

\item $T_{K_\Phi,E}$ is bounded on $L^2(E)$ for every regular curve $E$.

\item There exists $\eta\in(0,1]$ such that $T_{K_\Phi,E}$ is bounded on $L^2(E)$ for every $E=\gamma_g([0,1])$ with $g\in C([0,1])$ and $\norm{g}_\infty\le\eta$.
\end{enumerate}
If \textup{(i)} fails, then for every $\eta\in(0,1]$ there is $g\in C([0,1])$ with $\norm{g}_\infty\le\eta$ such that
\begin{equation}\label{eq:infinite-energy}
\energy_{K_\Phi}\bigl(\gamma_g([0,1])\bigr)=\infty.
\end{equation}
These curves are $1$-Ahlfors regular with an absolute regularity constant.
\end{theorem}

Condition \textup{(i)} means that $\Phi(u)\le C u^4$ for all sufficiently small $u>0$. By continuity, the same inequality holds on $[0,1]$ after increasing $C$. Hence $K_\Phi\le C K_4$, and thus $\textup{(i)} \Rightarrow \textup{(ii)}$ and $\textup{(i)} \Rightarrow \textup{(iii)}$ follow from the theorem of F\"assler and Orponen. The main part of our proof is the construction of a counterexample when \textup{(i)} fails: we will show that unboundedness can occur even on horizontal $C^1$ lifts of planar graphs with arbitrarily small planar slope.

For the kernels $K_\alpha$, we can choose the same curve for all $0<\alpha<4$. More precisely, we have the following theorem.

\begin{theorem}\label{thm:universal}
For every $\eta\in(0,1]$, there is a function $g\in C([0,1])$ with $\norm{g}_\infty\le\eta$ such that $E=\gamma_g([0,1])$ satisfies
\[
\energy_{K_\alpha}(E)=\infty \qquad\text{for every }0<\alpha<4.
\]
Consequently, $T_{K_\alpha,E}$ is unbounded on $L^2(E)$ for every $0<\alpha<4$, whereas it is bounded for every $\alpha\ge4$.
\end{theorem}

Theorem~\ref{thm:universal} addresses the exponent question following \cite[Theorem 1.7]{FO}, and extends the counterexamples of Chousionis, Li and Zhang~\cite{CLZ26} from $0<\alpha<2$ to the full range $0<\alpha<4$. Combined with the boundedness theorem of F\"assler and Orponen~\cite[Theorem 1.7]{FO} for $\alpha\ge4$, our result shows that the singular integral associated with $K_\alpha$ is $L^2$-bounded on every $1$-Ahlfors regular curve if and only if $\alpha\ge4$.

The $C^1$ conclusion also relates to a result of Chousionis, Li and Zimmerman~\cite[Theorem 1.1]{CLZi}. They proved that a $1$-dimensional CZ kernel which is uniformly $L^2$-bounded on horizontal lines defines bounded operators on compact regular $C^{1,\beta}$ curves in Carnot groups, for every $\beta\in(0,1]$. For every $\alpha>0$, the kernel $K_\alpha$ satisfies their assumptions. Thus the associated singular integral operators are $L^2$-bounded on every compact regular $C^{1,\beta}$ curve in $\mathbb H$, for $0<\beta\le 1$. However, Theorem~\ref{thm:universal} provides a compact regular horizontal $C^1$ curve on which these operators are unbounded for every $0<\alpha<4$. Therefore, $C^1$ regularity alone does not suffice for their boundedness conclusion.

We briefly describe the construction. We choose positive amplitudes $(a_n)_{n\ge 1}$ with $\sum_n a_n^2$ small and Rademacher functions $r_{k_n}$ at sufficiently separated scales. The function
\[
g=\eta\sin\!\left(\sum_{n=1}^\infty a_n r_{k_n}\right)
\]
is bounded by $\eta$, where the series is interpreted as an $L^2$ limit. At each selected scale, its dyadic averages retain oscillations comparable to $\eta a_n$ on a collection of intervals of total length bounded below. When estimating the energy $\energy_{K_\Phi}(E)$, the separation of scales allows us to control the contribution from all later oscillations to the signed area of the planar graph.

For each such interval, we obtain a lower bound for the vertical coordinate of $\gamma_g(s)^{-1}\gamma_g(t)$ on a rectangle of pairs $(s,t)$ whose area is a fixed fraction of the square of the interval length.
Rectangles coming from different scales are disjoint. Proposition~\ref{prop:energy} then gives
\[
\energy_{K_\Phi}\bigl(\gamma_g([0,1])\bigr) \ge c\sum_{n=1}^\infty\Phi(c_\eta\sqrt{a_n}),
\]
where $c>0$ is an absolute constant and $c_\eta>0$ depends only on $\eta$.
For $\Phi(u)=u^\alpha$, the sum on the right is a positive constant times $\sum_n a_n^{\alpha/2}$. When $\alpha<4$, this sum can diverge while $\sum_n a_n^2$ remains finite. More generally, if $\Phi(u)/u^4$ is unbounded near zero, we choose the amplitudes so that the same energy lower bound diverges.

The geometric background includes the traveling salesman theorems of Jones \cite{Jones1990} and Okikiolu \cite{Okikiolu1992} in Euclidean spaces, and related results of Ferrari, Franchi, and Pajot \cite{FFP2007} and Li and Schul \cite{LSCurvature2016,LSLength2016} in the Heisenberg group. Juillet \cite{Juillet} constructed a counterexample showing that the Euclidean square-summability condition on $\beta$-numbers need not hold for rectifiable curves in the Heisenberg group. Li \cite[Section~3]{Li} modified this construction to obtain divergence of the corresponding $\beta$-number sums for every exponent below four. 
Here we describe the curves through their planar slopes and estimate the kernel energy directly using explicit signed-area formulas. A methodological parallel appears in the work of Anceschi, Dietert, Guerand, Loher, Mouhot, and Rebucci \cite{ADGLMR} on hypoelliptic Poincar\'e inequalities. They refine earlier methods based on piecewise affine trajectories by constructing smooth trajectories through a control problem. Although the problems are different, both approaches choose the functions defining the curves so that the quantities needed in the estimates can be computed and controlled directly.

The planar slope constructed from the Rademacher functions is initially only bounded and measurable. To obtain a $C^1$ curve, we first approximate finite parts of the construction by smooth planar slopes while retaining arbitrarily large energy. We then place rescaled copies on disjoint intervals accumulating at $0$, choosing the planar slope bounds to tend to zero as the intervals approach $0$. The resulting planar slope $g$ is continuous, so $\gamma_g$ is a horizontal $C^1$ curve with infinite energy.

The paper is organized as follows. Section~\ref{2} establishes estimates for horizontal lifts and compares the auxiliary integral $I_\Phi(g)$ with the energy $\energy_{K_\Phi}(E)$. Section~\ref{3} constructs the bounded planar slope and estimates its dyadic averages. Section~\ref{4} proves the energy lower bound and chooses the amplitudes. Section~\ref{5} gives the smooth approximation and gluing arguments, records a consequence for logarithmic modifications of $K_4$, and completes the proofs of Theorems~\ref{thm:classification} and~\ref{thm:universal}.

\section{Horizontal lifts}\label{2}

In this section, we establish estimates for horizontal lifts and compare the auxiliary integral $I_\Phi(g)$ with the energy $\energy_{K_\Phi}(E)$.

The following lemma expresses the third coordinate of $\gamma_g(s)^{-1}\gamma_g(t)$ in terms of signed area and shows that $\gamma_g$ parametrizes a $1$-Ahlfors regular curve.

\begin{lemma}\label{lem:lift}
Suppose $\norm{g}_\infty\le\eta\le1$. For $0\le s<t\le1$, the third coordinate of $\gamma_g(s)^{-1}\gamma_g(t)$ is
\begin{equation}\label{eq:area}
A_g(s,t)=\int_s^t\left(u-\frac{s+t}{2}\right)g(u)\,du =\frac{t-s}{2}\big(f_g(s)+f_g(t)\big)-\int_s^t f_g(u)\,du.
\end{equation}
Furthermore,
\begin{equation}\label{eq:bilip}
|t-s|\le d(\gamma_g(s),\gamma_g(t))\le D_\eta|t-s|\le2|t-s|, \quad D_\eta=\big((1+\eta^2)^2+\eta^2/16\big)^{1/4}.
\end{equation}
For Borel $B\subset[0,1]$, 
\begin{equation}\label{eq:measure}
|B|\le\cH^1(\gamma_g(B))\le2|B|.
\end{equation}
In particular, $E=\gamma_g([0,1])$ is $1$-Ahlfors regular and $1\le\cH^1(E)\le2$.
\end{lemma}

\begin{proof}
The group law gives the third coordinate as
\[
z_g(t)-z_g(s)+\tfrac12\big(-s f_g(t)+t f_g(s)\big).
\]
Substituting \eqref{eq:lift} and integrating $u f_g'(u)$ by parts gives the second expression in \eqref{eq:area}; another integration by parts gives the first. Put $r=t-s$. Then
\[
|f_g(t)-f_g(s)|\le\eta r,\qquad |A_g(s,t)|\le\eta\int_s^t|u-(s+t)/2|\,du=\eta r^2/4.
\]
Substitution in $\|\cdot \|$ proves \eqref{eq:bilip}. The map $\gamma_g$ is $2$-Lipschitz and its inverse, the first-coordinate projection, is $1$-Lipschitz. The standard Hausdorff-measure inequality for Lipschitz maps proves \eqref{eq:measure}.

To prove that $E$ is $1$-Ahlfors regular, observe that $\gamma_g^{-1}(B(\gamma_g(t),r))$ is contained in $(t-r,t+r)\cap[0,1]$ and contains $(t-r/2,t+r/2)\cap[0,1]$.
For $0<r\le\diam E\le2$, \eqref{eq:measure} consequently gives
\[
r/2\le\cH^1(E\cap B(\gamma_g(t),r))\le4r.
\]
\end{proof}

It will be convenient to use an auxiliary integral
\begin{equation}\label{eq:param-energy}
I_\Phi(g)=\int_0^1\int_0^1 K_\Phi\big(\gamma_g(s)^{-1}\gamma_g(t)\big)\,dt\,ds.
\end{equation}
By \eqref{eq:measure}, $\energy_{K_\Phi}(E)\ge I_\Phi(g)$.
Moreover, if $E$ is compact and its energy $\energy_{K_\Phi}(E)\ge I_\Phi(g)=\infty$, positivity implies
\begin{equation}\label{eq:norm-test}
\sup_{\varepsilon>0}\norm{T_{K_\Phi,E,\varepsilon}}_{2\to2} \ge \sup_{\varepsilon>0} \frac{\langle T_{K_\Phi,E,\varepsilon}\1_E,\1_E\rangle}{\cH^1(E)} =\infty.
\end{equation}
Thus $T_{K_\Phi, E}$ would be unbounded on $L^2(E)$. Note no principal-value cancellation is involved in the obstruction.

\section{A bounded planar slope with prescribed oscillations}\label{3}

In this section, we construct a bounded measurable planar slope $g$ from a Rademacher series. We compute its dyadic averages explicitly and show that they retain oscillations of the prescribed size. These estimates will be used to control signed area in Section~\ref{4}; see also Remark~\ref{rem:signedarea}.

We use Lebesgue probability measure on $[0,1)$. For $k\ge1$ let
\[
r_k(t)=(-1)^{\lfloor2^k t\rfloor}.
\]
The functions $r_k$ are independent Rademacher functions. Let $\cD_k$ be the sigma-algebra generated by the dyadic intervals of length $2^{-k}$.

Fix positive numbers $(a_n)_{n\ge1}$ satisfying
\begin{equation}\label{eq:square-budget}
\sum_{n=1}^\infty a_n^2\le\frac1{64}.
\end{equation}
Choose integers $1\le k_1<k_2<\cdots$ so that, writing $\ell_n=2^{-(k_n-1)}$,
\begin{equation}\label{eq:sparse}
\ell_{n+1}\le\frac{a_n}{2048}\ell_n.
\end{equation}
There is no restriction on how rapidly $k_n$ grows, so this is possible.
Put
\begin{equation}\label{eq:S}
S_n=\sum_{j=1}^n a_jr_{k_j},\quad S_0=0, \qquad S=\lim_{n\to\infty}S_n\ \text{in }L^2(0,1), \qquad g=\eta\sin S.
\end{equation}
The $L^2$ limit exists by orthogonality. In particular $\norm{g}_\infty\le\eta$.
Define
\begin{equation}\label{eq:conditional}
c_n=\prod_{j>n}\cos a_j,\qquad g_n=\eta c_n\sin S_n.
\end{equation}

\begin{remark}
We set $g=\eta\sin S$ rather than $g=\eta S$ to ensure $\|g\|_\infty\le\eta$, since we only have an $L^2$ bound for $S$.
\end{remark}

The following lemma computes the dyadic averages $g_n$ of $g$ and shows that $g-g_n$ has mean zero on every dyadic interval of length $\ell_{n+1}$. It also shows that $g_n$ retains oscillations of size comparable to $\eta a_n$ on intervals whose total length is at least $15/16$.

\begin{lemma}\label{lem:conditional}
For every $n\ge1$,
\begin{equation}\label{eq:conditional-mean}
g_n=\E(g\mid\cD_{k_n})=\E(g\mid\cD_{k_{n+1}-1}).
\end{equation}
In particular $h_n=g-g_n$ has mean zero on every dyadic interval of length $\ell_{n+1}$, and $\norm{h_n}_\infty\le2\eta$.
On each dyadic interval $I$ of length $\ell_n$, $g_n$ has the form
\begin{equation}\label{eq:two-step}
g_n=C_I+b_I r_{k_n},\qquad b_I=\eta c_n\cos(S_{n-1}|_I)\sin a_n.
\end{equation}
The union of intervals $I$ on which $|S_{n-1}|\le1/2$ has measure at least $15/16$, and on each of these intervals $I$ with $|S_{n-1}|\le1/2$ we have
\begin{equation}\label{eq:b-lower}
b_I\ge\eta a_n/8.
\end{equation}
\end{lemma}

\begin{proof}
Fix $M>n$. Since $\cD_{k_n}$ is generated by the level sets of $r_1,\ldots,r_{k_n}$, the function $S_n$ is $\cD_{k_n}$-measurable, while $r_{k_{n+1}},\ldots,r_{k_M}$ are independent of $\cD_{k_n}$ and of one another. Each of these Rademacher functions takes the values $+1$ and $-1$ with equal probability.
Thus
\[
\begin{aligned}
\E(\sin S_M\mid\cD_{k_n}) &=\frac{1}{2^{M-n}} \sum_{\varepsilon_{n+1},\ldots,\varepsilon_M\in\{-1,1\}} \sin\left(S_n+\sum_{j=n+1}^M a_j\varepsilon_j\right)\\
&=\sin S_n\prod_{j=n+1}^M\cos a_j.
\end{aligned}
\]
The last equality follows by averaging over one sign at a time and using the identity
\[
\frac{\sin(u+a)+\sin(u-a)}{2}=\sin u\cos a.
\]

The same formula holds conditional on $\cD_{k_{n+1}-1}$, since no further selected sign is measurable there. The sine function is $1$-Lipschitz, so $\sin S_M\to\sin S$ in $L^2$. Conditional expectation is an $L^2$ contraction, proving \eqref{eq:conditional-mean}.

Each $a_n\le1/8$. Also $1-\cos a_n\le a_n^2/2$, and hence finite-product estimates followed by a limit give
\begin{equation}\label{eq:product-lower}
c_n\ge1-\tfrac12\sum_{j>n}a_j^2\ge127/128>1/2.
\end{equation}
On $I$, all the signs appearing in $S_{n-1}$ are constant. Expanding $\sin(S_{n-1}+a_nr_{k_n})$ proves \eqref{eq:two-step}.
Chebyshev's inequality gives
\[
|\{|S_{n-1}|>1/2\}|\le4\norm{S_{n-1}}_2^2 =4\sum_{j<n}a_j^2\le1/16.
\]
On the complementary union of intervals, $\cos S_{n-1}\ge1/2$.
Together with $\sin a_n\ge a_n/2$ and \eqref{eq:product-lower}, this proves \eqref{eq:b-lower}.
\end{proof}

\section{The energy lower bound}\label{4}

In this section, we obtain a lower bound for $I_\Phi(g)$ for the planar slope constructed in Section~\ref{3}. We then choose the amplitudes so that $I_\Phi(g)=\infty$. At this stage, $g$ is only bounded and measurable; we obtain a continuous planar slope in Section~\ref{5}.

The following proposition bounds $I_\Phi(g)$ from below by a sum determined by the amplitudes $(a_n)_{n\ge 1}$. Continuity of $\Phi$ is not needed for this estimate. 

\begin{proposition}[Lower bound for the auxiliary integral]\label{prop:energy}
Let $\Phi:[0,1]\to[0,\infty)$ be nondecreasing with $\Phi(0)=0$.
Given $(a_n)_{n\ge 1}$ satisfying \eqref{eq:square-budget} and $0<\eta\le1$, the function $g$ constructed in Section 3 satisfies
\begin{equation}\label{eq:main-lower}
I_\Phi(g)\ge\frac{15}{512} \sum_{n=1}^\infty\Phi\!\left(\frac{\sqrt\eta}{32}\sqrt{a_n}\right).
\end{equation}
\end{proposition}

\begin{proof}
Call the intervals in the last part of Lemma~\ref{lem:conditional} good.
For a good $I=[v,v+\ell_n)$, put $m=v+\ell_n/2$ and
\[
L_I=[v,v+\ell_n/4],\qquad R_I=[v+3\ell_n/4,v+\ell_n].
\]
For $s\in L_I$, $t\in R_I$, the constant term in \eqref{eq:two-step} has zero contribution to \eqref{eq:area}, and direct integration gives
\begin{equation}\label{eq:coarse-area}
 A_{g_n}(s,t)=b_I\left( \Big(m-\frac{s+t}{2}\Big)^2 - \Big(\frac{s-t}{2}\Big)^2\right)=-b_I(m-s)(t-m).
\end{equation}
Consequently,
\begin{equation}\label{eq:coarse-lower}
 |A_{g_n}(s,t)|\ge\eta a_n\ell_n^2/128.
\end{equation}

We now control all finer oscillations uniformly on this entire rectangle $L_I\times R_I$.
Let $H_n(u)=\int_0^u h_n(w)\,dw$. By the mean $0$ property in Lemma~\ref{lem:conditional}, every dyadic interval of length $\ell_{n+1}$ contributes zero to $H_n$. Thus
\[
 \norm{H_n}_\infty=\norm{\int_{\lfloor \frac{u}{l_{n+1}} \rfloor l_{n+1}}^u h_n(w)\,dw}_\infty \le2\eta\ell_{n+1}.
\]
Integration by parts, with $r=t-s$, gives
\begin{align*}
 A_{h_n}(s,t) &=\frac r2\big(H_n(t)+H_n(s)\big)-\int_s^t H_n(u)\,du,\\
 |A_{h_n}(s,t)|&\le2r\norm{H_n}_\infty \le4\eta\ell_n\ell_{n+1} \le\eta a_n\ell_n^2/512.
\end{align*}
Combining this with \eqref{eq:coarse-lower} yields
\begin{equation}\label{eq:area-lower}
 |A_g(s,t)|\ge\eta a_n\ell_n^2/256.
\end{equation}
By Lemma~\ref{lem:lift},
\[
 \ell_n/2\le t-s\le\ell_n, \qquad \|\gamma_g(s)^{-1}\gamma_g(t))\|\le2\ell_n.
\]
It follows that
\begin{equation}\label{eq:angle-lower}
 \theta(\gamma_g(s)^{-1}\gamma_g(t)) \ge\frac{\sqrt\eta}{32}\sqrt{a_n}.
\end{equation}
Using monotonicity of $\Phi$ and $|L_I\times R_I|=\ell_n^2/16$, we obtain
\begin{equation}\label{eq:one-rectangle}
 \int_{L_I}\int_{R_I} K_\Phi(\gamma_g(s)^{-1}\gamma_g(t))\,dt\,ds \ge\frac{\ell_n}{32} \Phi\!\left(\frac{\sqrt\eta}{32}\sqrt{a_n}\right).
\end{equation}
The total length of good intervals at level $n$ is at least $15/16$.
Their rectangles are pairwise disjoint. Rectangles from different levels are disjoint as well: the separation of $s$ and $t$ lies in $[\ell_n/2,\ell_n]$, and \eqref{eq:sparse} implies $\ell_{n+1}<\ell_n/2$. Summing \eqref{eq:one-rectangle} proves
\eqref{eq:main-lower}. 
\end{proof}

\begin{remark}\label{rem:finite-rectangles}
For each $M\ge1$, let $F_M$ be the union of the rectangles $L_I\times R_I$ used in the proof for $1\le n\le M$.
This is a finite union, and every $(s,t)\in F_M$ satisfies
\[
 t-s\ge\ell_M/2>0.
\]
The estimates in the proof give
\[
 \iint_{F_M} K_\Phi\bigl(\gamma_g(s)^{-1}\gamma_g(t)\bigr)\,dt\,ds \ge \frac{15}{512} \sum_{n=1}^M \Phi\!\left(\frac{\sqrt\eta}{32}\sqrt{a_n}\right).
\]
Thus, if the series in \eqref{eq:main-lower} diverges, we can make the integral over $F_M$ arbitrarily large by choosing $M$ large enough, while $F_M$ remains a positive distance from the diagonal. We will use this fact in the smooth approximation argument in Section~\ref{5}.
\end{remark}

For the power kernels $K_\alpha$, the following corollary gives one bounded planar slope $g$ for which $\gamma_g([0,1])$ has infinite energy for every $0<\alpha<4$.

\begin{corollary}\label{cor:lipschitz-universal}
For each $\eta\in(0,1]$, one bounded planar slope $g$, with $\norm{g}_\infty\le\eta$, gives infinite energy for all $K_\alpha$, $0<\alpha<4$.
\end{corollary}

\begin{proof}
Choose
\[
 a_n=\frac{c}{\sqrt n\log(n+1)},
\]
where $c>0$ is small enough for \eqref{eq:square-budget}.
Then $\sum a_n^2<\infty$, whereas $\sum a_n^q=\infty$ for every $0<q<2$. Apply Proposition~\ref{prop:energy} with $\Phi(u)=u^\alpha$.
The choices of $(a_n)$, $(k_n)$ and $g$ in \eqref{eq:S} do not depend on $\alpha$.
\end{proof}

\begin{remark}\label{rem:signedarea}
The quantity $A_g(s,t)$ has a simple geometric meaning.
Let $\lambda_{s,t}$ be the affine function satisfying $\lambda_{s,t}(s)=f_g(s)$ and $\lambda_{s,t}(t)=f_g(t)$.
By \eqref{eq:area},
\[
A_g(s,t)=\int_s^t\bigl(\lambda_{s,t}(u)-f_g(u)\bigr)\,du.
\]
Thus $A_g(s,t)$ is the signed area between the planar curve and the line segment joining its endpoints. Regions below this segment contribute positively, and regions above it contribute negatively; see Figure~\ref{fig:signed-area}.

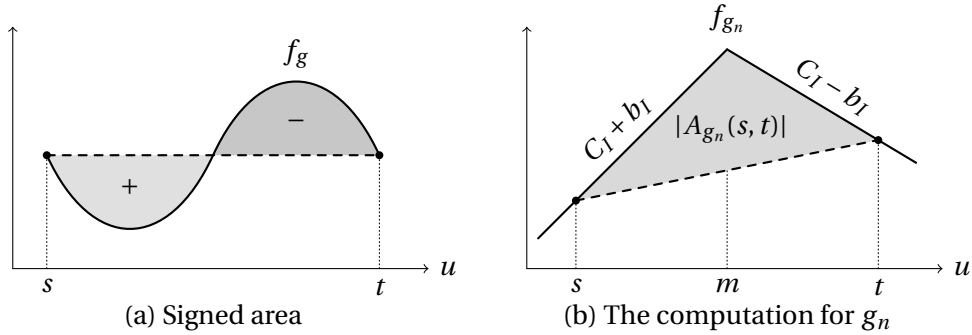
\begin{figure}[htbp]
\centering
\begin{tikzpicture}[x=1cm,y=1cm,font=\small,
  line cap=round,line join=round]

\begin{scope}
  \fill[black!12]
    (0.6,1.5) .. controls (1.2,0.2) and (2.2,0.2) .. (2.8,1.5)
    -- cycle;
  \fill[black!22]
    (2.8,1.5) .. controls (3.4,2.8) and (4.4,2.8) .. (5,1.5)
    -- cycle;
  \draw[->] (0.15,0) -- (5.65,0) node[right] {$u$};
  \draw[->] (0.15,0) -- (0.15,3.2);
  \draw[densely dotted] (0.6,0) -- (0.6,1.5);
  \draw[densely dotted] (5,0) -- (5,1.5);
  \node[below] at (0.6,0) {$s$};
  \node[below] at (5,0) {$t$};
  \draw[thick,dashed] (0.6,1.5) -- (5,1.5);
  \draw[thick]
    (0.6,1.5) .. controls (1.2,0.2) and (2.2,0.2) .. (2.8,1.5)
    .. controls (3.4,2.8) and (4.4,2.8) .. (5,1.5);
  \fill (0.6,1.5) circle (1.5pt);
  \fill (5,1.5) circle (1.5pt);
  \node at (1.7,1.08) {$+$};
  \node at (3.9,1.92) {$-$};
  \node[above] at (3.9,2.52) {$f_g$};
  \node at (2.8,-0.65) {(a) Signed area};
\end{scope}

\begin{scope}[xshift=6.8cm]
  \coordinate (P) at (0.8,0.9);
  \coordinate (Q) at (4.8,1.7);
  \coordinate (V) at (2.8,2.9);
  \fill[black!15] (P) -- (V) -- (Q) -- cycle;
  \draw[->] (0.15,0) -- (5.65,0) node[right] {$u$};
  \draw[->] (0.15,0) -- (0.15,3.2);
  \draw[densely dotted] (0.8,0) -- (P);
  \draw[densely dotted] (2.8,0) -- (2.8,1.3);
  \draw[densely dotted] (4.8,0) -- (Q);
  \node[below] at (0.8,0) {$s$};
  \node[below] at (2.8,0) {$m$};
  \node[below] at (4.8,0) {$t$};
  \draw[thick] (0.3,0.4) -- (V)
    node[midway,above,sloped] {$C_I+b_I$};
  \draw[thick] (V) -- (5.3,1.4)
    node[midway,above,sloped] {$C_I-b_I$};
  \draw[thick,dashed] (P) -- (Q);
  \fill (P) circle (1.5pt);
  \fill (Q) circle (1.5pt);
  \node at (2.85,1.85) {$|A_{g_n}(s,t)|$};
  \node[above] at (2.8,2.98) {$f_{g_n}$};
  \node at (2.8,-0.65) {(b) The computation for $g_n$};
\end{scope}
\end{tikzpicture}
\caption{Signed area between the planar curve and the line segment connecting its endpoints.
In (a), regions below the line segment connecting the endpoints contribute positively to $A_g(s,t)$, and regions above contribute negatively.
In (b), the two segments have planar slopes $C_I+b_I$ and $C_I-b_I$.
The vertical distance from the vertex at $u=m$ to the line segment joining the endpoints is $|f_{g_n}(m)-\lambda_{s,t}(m)|=|2b_I(m-s)(t-m)/(t-s)|$.
Hence the shaded triangle has area $|b_I(m-s)(t-m)|$, and $|A_{g_n}(s,t)|=|-b_I(m-s)(t-m)|$, where $m$ is the midpoint of $I$.}
\label{fig:signed-area}
\end{figure}

Our construction is related to those of Juillet~\cite{Juillet}, Li~\cite{Li}, and Chousionis, Li and Zhang~\cite{CLZ26}. We use Rademacher functions and conditional expectations to obtain explicit formulas for the dyadic averages of the planar slope. These formulas allow us to compute the signed area at each selected scale, while the separation of scales controls the contribution of finer oscillations. We can then choose the amplitudes and scales to obtain the required divergence of $I_\Phi(g)$.
\end{remark}

To treat general $\Phi$, we need to choose the amplitudes so that the sum in \eqref{eq:main-lower} diverges while $\sum_n a_n^2$ remains small. The following lemma provides the required choice of amplitudes.

\begin{lemma}[Selecting amplitudes]\label{lem:selection}
Suppose $\psi:(0,a_0]\to[0,\infty)$ satisfies $\limsup_{a\downarrow0}\psi(a)/a^2=\infty$.
For every $B>0$ there is a sequence of positive numbers $(a_n)_{n\ge 1}$ such that $a_n\le a_0$, $\sum_n a_n^2\le B$, and $\sum_n\psi(a_n)=\infty$.
\end{lemma}

\begin{proof}
For each $j\ge1$, choose $b_j>0$ so small that
\[
 b_j\le\min(a_0,2^{-j}),\qquad b_j^2\le B2^{-j-1},\qquad \psi(b_j)\ge2^{2j}b_j^2.
\]
Use $M_j=\lfloor B2^{-j}/b_j^2\rfloor$ copies of $b_j$, i.e.,
\[
(a_n)_{n\ge1}=\bigl(
\underbrace{b_1,\ldots,b_1}_{M_1\text{ copies}},
\underbrace{b_2,\ldots,b_2}_{M_2\text{ copies}},
\ldots\bigr).
\]
Then $M_jb_j^2\le B2^{-j}$, while $b_j^2 \le B2^{-j-1}$ implies $M_j\ge B2^{-j}/(2b_j^2)$. Thus
\[
 \sum_jM_jb_j^2\le B,\qquad M_j\psi(b_j)\ge(B/2)2^j.
\]
Concatenating these finite sequences gives the required sequence.
\end{proof}

The following corollary combines Lemma~\ref{lem:selection} with Proposition~\ref{prop:energy} to obtain a bounded planar slope $g$ satisfying $I_\Phi(g)=\infty$. We include continuity of $\Phi$ in the hypotheses because it will be needed for the smooth approximation in Section~\ref{5}.

\begin{corollary}\label{cor:profile-obstruction}
If $\Phi$ is continuous and nondecreasing on $[0,1]$, $\Phi(0)=0$, and $\Phi(u)/u^4$ is unbounded as $u\downarrow0$, then for every $\eta\in(0,1]$ there is a bounded planar slope $g$, with $\norm{g}_\infty\le\eta$, for which $ I_\Phi(g)=\infty$.
\end{corollary}

\begin{proof}
Set $c_\eta=\sqrt\eta/32$ and $\psi(a)=\Phi(c_\eta\sqrt a)$ for $0<a\le1/8$. With $u=c_\eta\sqrt a$,
\[
 \frac{\psi(a)}{a^2}=c_\eta^4\frac{\Phi(u)}{u^4}.
\]
Apply Lemma~\ref{lem:selection} with $B=1/64$, and then Proposition~\ref{prop:energy}.
\end{proof}

\section{Smooth approximation and gluing}\label{5}

In this section, we construct a continuous planar slope $g$ with $I_\Phi(g)=\infty$. We first obtain smooth functions with arbitrarily large auxiliary integral and then combine rescaled copies of these functions. This gives the required $1$-Ahlfors regular horizontal $C^1$ curves and completes the proofs of Theorems~\ref{thm:classification} and~\ref{thm:universal}.

The following lemma shows that the auxiliary integral can remain arbitrarily large after smoothing the planar slope. 
Here continuity of $\Phi$ allows us to preserve a large integral over the finite union of rectangles in Remark~\ref{rem:finite-rectangles}.

\begin{lemma}[Smooth functions with arbitrarily large auxiliary integral]\label{lem:smooth}
Under the hypotheses of Corollary~\ref{cor:profile-obstruction}, for every $0<\eta\le1$ and $M>0$ there exists $b\in C_c^\infty(0,1)$ such that $\norm b_\infty\le\eta$ and $ I_\Phi(b)>M$.
\end{lemma}

\begin{proof}
Choose the planar slope $g$ in Corollary~\ref{cor:profile-obstruction}.
By Remark~\ref{rem:finite-rectangles}, there is a finite union $F$ of its testing rectangles, separated from the diagonal, such that the auxiliary integral over $F$ exceeds $2M$.

Choose $b_j\in C_c^\infty(0,1)$ satisfying $\norm{b_j}_\infty\le\eta$ and $b_j\to g$ in $L^1(0,1)$.
Such an approximation follows by cutting off near the endpoints and convolving with a nonnegative approximate identity.
The definitions give
\[
 \norm{f_{b_j}-f_g}_\infty\le\norm{b_j-g}_1, \qquad \norm{z_{b_j}-z_g}_\infty\le\norm{b_j-g}_1.
\]
Thus $\gamma_{b_j}\to\gamma_g$ uniformly. On $F$ all the first-coordinate differences are bounded away from zero. The kernel $K_\Phi$ is continuous off the identity, and the group differences $\gamma_g(s)^{-1}\gamma_g(t)$ under consideration lie in a fixed compact set away from the identity. The kernel integrands therefore converge uniformly on $F$. For sufficiently large $j$, their integrals over $F$ exceed $M$. Since the full auxiliary integral is nonnegative, $b=b_j$ has the asserted property.
\end{proof}

The following proposition combines rescaled copies of the smooth functions from Lemma~\ref{lem:smooth} into one continuous planar slope with infinite auxiliary integral. We place the copies on disjoint intervals accumulating at $0$ and choose their maximum absolute values to tend to zero as the intervals approach $0$.

\begin{proposition}[Gluing]\label{prop:gluing}
If $\Phi$ satisfies the hypotheses of Corollary~\ref{cor:profile-obstruction}, then for every $0<\eta\le1$ there is $g\in C([0,1])$ with $\norm g_\infty\le\eta$ and $ I_\Phi(g)=\infty$.
\end{proposition}

\begin{proof}
For each $j\ge1$, apply Lemma~\ref{lem:smooth} with $\eta_j=\eta2^{-j}$ and $M_j=2^{2j}$, obtaining $b_j$.
Let
\[
 x_j=2^{-j},\qquad \lambda_j=2^{-j-2},\qquad  I_j=(x_j,x_j+\lambda_j).
\]
These intervals are disjoint and accumulate only at zero. Define
\begin{equation}\label{eq:glued-slope}
 g(t)=
 \begin{cases}
 b_j((t-x_j)/\lambda_j),&t\in I_j,\\
 0,&t\notin\bigcup_j I_j.
 \end{cases}
\end{equation}
Each $b_j$ vanishes near its endpoints, and $\norm{b_j}_\infty\le\eta_j\to0$. Therefore $g$ is continuous, including at zero, and $\norm g_\infty\le\eta$.

Put $P_j=\gamma_g(x_j)$. Direct substitution in \eqref{eq:lift} gives
\begin{equation}\label{eq:scaled-block}
 \gamma_g(x_j+\lambda_j u)=P_j\cdot\delta_{\lambda_j}\gamma_{b_j}(u) \qquad(0\le u\le1).
\end{equation}
Indeed the equality holds for the first two coordinates. For the third coordinate, since left translation preserves horizontality, both sides are horizontal curves. And since both sides have the same starting point when $u=0$, the equality holds for the third coordinate.

Since $K_\Phi(\delta_\lambda p)=\lambda^{-1}K_\Phi(p)$,
\eqref{eq:scaled-block} and a change of variables imply
\[
 \int_{I_j}\int_{I_j} K_\Phi(\gamma_g(s)^{-1}\gamma_g(t))\,dt\,ds=\int_0^1\int_0^1 \lambda_j^{-1} K_\Phi\!\left(\gamma_{b_j}(u)^{-1}\gamma_{b_j}(v)\right)\lambda_j^2\,dv\,du  =\lambda_j I_\Phi(b_j)>2^{j-2}.
\]
The squares $I_j\times I_j$ are disjoint. Summing and using positivity gives $ I_\Phi(g)=\infty$.
\end{proof}

The following corollary gives counterexamples for logarithmic modifications of $K_4$. These examples will also yield a single curve for all $0<\alpha<4$ in Theorem~\ref{thm:universal}.

\begin{corollary}[Logarithmic strengthening of $K_4$]\label{cor:log}
Let $\beta>0$, and let $\Phi_\beta$ be a continuous nondecreasing extension to $[0,1]$ of
\[
 \Phi_\beta(u)=u^4\big(\log(e/u)\big)^\beta
\]
for sufficiently small $u>0$, with $\Phi_\beta(0)=0$.
There is a compact $1$-Ahlfors regular horizontal $C^1$ curve on which $T_{K_{\Phi_\beta}}$ is unbounded. Its planar slope can be arbitrarily small.
\end{corollary}

\begin{proof}
The given $\Phi$ is increasing near zero, so such an extension exists. Its ratio to $u^4$ tends to infinity. Apply Theorem~\ref{thm:classification}.
\end{proof}

\begin{proof}[Proof of Theorem~\ref{thm:universal}]
In Corollary~\ref{cor:log}, fix $\beta$, there exists a compact $1$-Ahlfors regular horizontal $C^1$ curve $E$ such that $T_{K_{\Phi_\beta},E}$ is $L^2$-unbounded. Since $\Phi_\beta \lesssim u^\alpha$ for every $\alpha \in (0,4)$, $T_{K_\alpha, E}$ is $L^2$-unbounded on $E$ for every $\alpha \in (0,4)$.
\end{proof}

\begin{proof}[Proof of Theorem~\ref{thm:classification}]
If \textup{(i)} holds, continuity and boundedness of $\Phi$ on $[0,1]$
give a finite constant $C_\Phi$ such that
\[
 \Phi(u)\le C_\Phi u^4\quad(0\le u\le1).
\]
Hence $0\le K_\Phi\le C_\Phi K_4$, and for every $\varepsilon>0$,
\[
 |T_{K_\Phi,E,\varepsilon}h|
 \le C_\Phi T_{K_4,E,\varepsilon}|h|.
\]
The known $K_4$ estimate proves \textup{(ii)}. Lemma~\ref{lem:lift}
gives \textup{(ii)}$\Rightarrow$\textup{(iii)}.
If \textup{(i)} fails, Proposition~\ref{prop:gluing} gives, for the
fixed $\eta$ in \textup{(iii)}, a continuous planar slope with infinite
auxiliary integral $I_\Phi(g)$. Equations \eqref{eq:measure} and
\eqref{eq:norm-test} contradict \textup{(iii)}. The same construction
works for each $\eta\in(0,1]$ and gives \eqref{eq:infinite-energy}.
\end{proof}

\begin{remark}
It is proven in \cite[Theorem 1.7]{FO} that the operator associated to $K_4$ is bounded on every $1$-Ahlfors regular curve.
The norm convention does not affect this assertion. That paper~\cite{FO} uses
\[
 \|(x,y,z)\|_{FO}=\big((x^2+y^2)^2+16z^2\big)^{1/4}, \qquad \|p\| \le \|p\|_{FO} \le2\|p\|.
\]
The associated Hausdorff measures satisfy $\cH_{\|\cdot \|}^1\le\cH_{\|\cdot \|_{FO}}^1\le2\cH_{\|\cdot \|}^1$, and regularity is preserved.
Also $|z|^2/\|p\|^5\le32|z|^2/\|p\|_{FO}^5$ and $\{\|p^{-1}q\|>\varepsilon\}\subset \{\|p^{-1}q\|>\varepsilon\}$.
For nonnegative inputs, we have pointwise comparison of truncated singular integral operators. For signed and complex-valued $f$, $|T_\varepsilon f|\le T_\varepsilon |f|$, so the same estimate applies.
\end{remark}

\section*{Statements and Declarations}

\noindent\textbf{Competing interests.}
The author has no competing interests to declare that are relevant
to the content of this article.

\medskip
\noindent\textbf{Data availability.}
No datasets were generated or analysed during the current study.

\medskip
\noindent\textbf{Use of AI assistance.}
ChatGPT was consulted for feedback on the manuscript and for
assistance with the wording of the literature comparison and
submission declarations, as well as with the preparation of the TikZ figure. The mathematical statements and proofs
remain as in the author's original manuscript. The author takes
responsibility for the content of the manuscript.

\bibliographystyle{plain}
\bibliography{ref}

\end{document}